\documentclass[onefignum,onetabnum]{siamonline250211}

\usepackage{amsfonts}

\usepackage{enumitem}
\setlist[enumerate]{leftmargin=.5in}
\setlist[itemize]{leftmargin=.5in}

\usepackage{amssymb}
\usepackage{amsopn}
\usepackage{mathtools}
\usepackage{hyperref}
\usepackage{needspace}
\usepackage{tikz}
\usetikzlibrary{arrows.meta}
\newsiamremark{remark}{Remark}

\headers{Balanced Factors and Strategy Scaling}{H. Hu, J. Wang, and M. Wang}

\title{Algebraic Degree of Network Games: Balanced Factors and Strategy Scaling\thanks{Preprint, version 2.0. \today.\funding{This work was funded by the National Natural Science Foundation of China (Grant No. 72243005), the National Key Research and Development Program of China (Grant No. 2020YFA0608602), the major project of philosophy and social science research in colleges and universities of Jiangsu Province (2024SJZD129), the major project of Basic Science (Natural science) research in colleges and universities of Jiangsu Province (24KJA110002), Special Science and Technology Innovation Program for Carbon Peak and Carbon Neutralization of Jiangsu Province (Grant No. BE2022612) and Open project Fund for Ministry of Education Key Laboratory of NSLSCS (202408). }}}

\author{
Hangkun Hu\thanks{Ministry of Education Key Laboratory of NSLSCS, School of Mathematical Sciences, Nanjing Normal University, Nanjing 210023, China (\email{250901005@njnu.edu.cn}).}
\and
Jingyi Wang\thanks{Ministry of Education Key Laboratory of NSLSCS, School of Mathematical Sciences, Nanjing Normal University, Nanjing 210023, China (\email{wjy05100@gmail.com}).}
\and
Minggang Wang\thanks{Ministry of Education Key Laboratory of NSLSCS, School of Mathematical Sciences, Nanjing Normal University, Nanjing 210023, China; Department of Mathematics, Nanjing Normal University Taizhou College, Taizhou 225300, Jiangsu, China (\email{05424@njnu.edu.cn}). Corresponding author.}
}

\ifpdf
\hypersetup{
  pdftitle={Algebraic Degree of Network Games: Balanced Factors and Strategy Scaling},
  pdfauthor={Hangkun Hu, Jingyi Wang, and Minggang Wang}
}
\fi

\DeclareMathOperator{\perm}{perm}
\DeclareMathOperator{\conv}{conv}

\begin{document}

\maketitle

\begin{abstract}
The algebraic degree of a network game is the generic number of isolated complex-torus solutions of its polynomial indifference system.  We recast the classical semi-mixed coefficient formula as an edge-marked player-level polynomial whose monomials are balanced directed multigraphs with prescribed in- and out-degrees.  This representation separates the intrinsic counting problem from strategy-label refinements: exact evaluation remains $\#\mathrm P$-complete, but for a fixed number of players it is polynomial in the numerical strategy dimensions, while nonvanishing is decided by a capacitated flow test.  It also characterizes inclusion-minimal positive-degree supports as integral transportation forests and yields a sharp $2N-1$-arc positive core.  Exact-support coefficients form connected transportation fibers and give a nonnegative support calculus.  Under proportional strategy growth $m\mathbf k$, a lattice local limit theorem expresses the first-order degree asymptotic through the capacity and maximum-entropy flow of the essential support.  All counts are generic and complex; reality and simplex feasibility remain payoff-dependent.
\end{abstract}



\section{Introduction}
\label{sec:intro}

Network games model strategic interactions with local dependence.  An arc $j\to i$ in the directed dependency graph records that player $i$'s payoff may depend on player $j$'s action \cite{KearnsLittmanSingh2001}.  A totally mixed equilibrium satisfies a square polynomial indifference system together with strict simplex inequalities.  We study its \emph{algebraic degree}: the generic number of isolated complex-torus solutions, counted with multiplicity.  It bounds, but does not equal, the number of real feasible equilibria.

For complete support, McKelvey and McLennan obtained the sharp count by constrained set partitions \cite{McKelveyMcLennan1997}.  Datta expressed the generic count through a permanent and incorporated graphical structural zeros \cite{Datta2003,DattaPolynomialGraphs2006,Datta2010}; Bernstein's theorem gives the equivalent mixed-volume formulation \cite{Bernshtein1975,HuberSturmfels1995}.  Semi-mixed multihomogeneous counts also admit classical B\'ezout and MacMahon generating-function descriptions \cite{MacMahon1915,EmirisVidunas2014,Emiris2017}.  The root-count identity used here is therefore foundational.  Our contribution is to retain the player blocks and mark the dependency arcs, thereby turning that identity into a support-sensitive transportation invariant.

A complementary vector-bundle formulation realizes the equilibrium scheme as the zero scheme of a section and obtains its degree from a top Chern class \cite{AboPortakalSodomaco2026}.  We record the direct intersection-theoretic interface in Remark~\ref{rem:intersection-theoretic-interface}.  Related undirected graphical models impose conditional-independence constraints on joint distributions rather than the directed payoff dependence considered here \cite{PortakalSendraArranz2025,BouyerPortakalSendraArranz2025}.

Write $\mathbf k=(k_1,\ldots,k_N)$ and attach a variable $x_{ji}$ to every potential arc $j\to i$.  The central object is
\[
\Phi_{\mathbf k}(x)
=[z^{\mathbf k}]
\prod_{i=1}^N\left(\sum_{j\ne i}x_{ji}z_j\right)^{k_i}.
\]
The paper studies the invariant $(E,\mathbf k)\mapsto\mathcal D(E,\mathbf k)$ through changes of support $E$ and proportional dilation of $\mathbf k$.  Its structural conclusions are summarized as follows.

\medskip
\noindent\textbf{Theorem A.}
\emph{Let $G$ be a directed dependency graph with arc set $E(G)$.  Then:}
\begin{enumerate}[label=\textup{(\roman*)}]
\item \emph{the generic degree is the support specialization of $\Phi_{\mathbf k}$ and the weighted count of balanced $\mathbf k$-factors,}
\begin{equation}
\label{eq:intro-network-degree}
\mathcal D(G,\mathbf k)
=\Phi_{\mathbf k}(\mathbf 1_{E(G)})
=\sum_{a\in\mathfrak E_{\mathbf k}(E(G))}
\prod_{i=1}^N\frac{k_i!}{\prod_{j\ne i}a_{ji}!}.
\end{equation}
\item \emph{exact evaluation is $\#\mathrm P$-complete already for $\mathbf k=\mathbf1$, whereas a truncated coefficient algorithm has $\prod_j(k_j+1)$ states and is polynomial in the numerical strategy dimensions for fixed $N$;}
\item \emph{positivity is equivalent to the capacitated Hall inequalities and is decidable by maximum flow.  Inclusion-minimal positive supports are precisely positive integral transportation forests.  Every positive-degree support contains such a core with at most $2N-1$ arcs, and this bound is sharp.}
\end{enumerate}

Here $\mathfrak E_{\mathbf k}(E)$ denotes the integral arc multiplicities supported on $E$ with in- and out-degree $k_i$ at every vertex.  The generic realization is \cref{prop:generic-realization}; exact-support refinements give monotonicity, supermodularity, residual Hall tests, cycle moves, and an Ehrhart interpretation.

For proportional growth, let $F^*=E_{\mathbf k}^*(E(G))$ be the set of arcs that are positive in some feasible transportation flow, and define
\[
\operatorname{cap}_{\mathbf k}(F^*)
=\inf_{z>0}
\frac{\prod_i(\sum_{j\to i\in F^*}z_j)^{k_i}}{\prod_jz_j^{k_j}}.
\]
\medskip
\noindent\textbf{Theorem B.}
\emph{Suppose $\mathcal D(G,\mathbf k)>0$, put $F^*=E_{\mathbf k}^*(E(G))$ and $c^*=c(F^*)$, and let $\widehat\Sigma^*$ be the reduced covariance matrix defined in \eqref{eq:reduced-covariance}.  Then}
\[
\mathcal D(G,m\mathbf k)
\sim
\frac{\operatorname{cap}_{\mathbf k}(F^*)^m}
{(2\pi m)^{(N-c^*)/2}\sqrt{\det\widehat\Sigma^*}}.
\]
The proof also shows that balanced factors using every arc of $F^*$ contribute asymptotic proportion one.

The capacity is the exponential of the entropy of the unique matrix-scaled balanced flow.  The proof combines standard capacity and local-limit methods \cite{Gurvits2008,Elfving1980,Barvinok2009,BarvinokHartigan2010} with essential-support reduction and the graph-dependent flow lattice.  For complete dependency graphs, Vidunas already derived the equal-dimension asymptotic and studied unequal proportional growth for three and four players \cite[Section~6]{Vidunas2017}; the equal-dimension formula is also recalled in \cite[Remark~2.12]{AboPortakalSodomaco2026}.  Theorem~B gives a single support-dependent formulation for arbitrary fixed directed networks and positive integer strategy vectors, including supports whose feasible flows lie on a proper face.  Section~\ref{sec:known-asymptotics} recovers the classical diagonal constant explicitly, and Section~\ref{sec:sparse-example} illustrates the sparse case.

The structural results use classical transportation, Hall, and Boolean-lattice tools; the contribution is their joint application to the edge-marked degree, with explicit minimal supports and a first-order scaling law on the essential support.  Section~7 adds factorization over strongly connected components.  All counts concern generic isolated complex roots; reality, simplex feasibility, and dynamics require additional payoff-specific conditions.

\paragraph{Version history}
This manuscript substantially revises and supersedes \href{https://arxiv.org/abs/2604.17741v1}{arXiv:2604.17741v1}.  The previous version incorrectly claimed a termwise identification between weighted cycle-cover contributions and tropical intersection multiplicities.  That claim has been removed.  The present treatment builds on Datta's classical permanent formula and the BKK mixed-volume interpretation, distinguishing the factorial normalization arising from strategy-label overcounting from lattice intersection multiplicities.  The manuscript is reorganized around balanced factors, support structure, and proportional strategy growth.  The discussion of real roots and feasible totally mixed equilibria has also been corrected: the generic complex algebraic degree does not by itself determine either count.

\section{Network Games and Polynomial Systems}
\label{sec:network_games}

\subsection{The Network Game Model}
\label{sec:model}

Consider a finite game with player set \(\mathcal I=\{1,\ldots,N\}\).  Player \(i\) has pure-strategy set
\[
S_i=\{0,1,\ldots,m_i\},
\qquad
k_i=|S_i|-1.
\]
We assume throughout that $N\ge2$ and $k_i\ge1$ for every player.  A player with only one pure strategy contributes no indifference equation and can be removed before forming the dependency system.
The directed graph \(G_{\mathrm{player}}=(V,E)\) records payoff dependence: an arc \(j\to i\) indicates that the payoff of player \(i\) may depend on the strategy of player \(j\).  With
\[
\mathcal N_i=\{j:j\to i\in E\}\cup\{i\},
\qquad
B_i=\mathcal N_i\setminus\{i\},
\]
the utility function \(u_i\) depends only on the strategy profile indexed by \(\mathcal N_i\).  We allow an arbitrary local payoff table on \(\prod_{j\in\mathcal N_i}S_j\); in particular, the model is not restricted to pairwise-additive interactions.  For matrices indexed by players, we use the receiver--row convention
\[
(A_G)_{ij}=1\quad\Longleftrightarrow\quad j\to i\in E(G).
\]
This is the transpose of the source--row convention.  Since the permanent is invariant under transposition, the choice does not affect any degree or cycle-cover count below.

\subsection{Algebraic Formulation}
\label{sec:algebraic_formulation}

A mixed strategy of player \(i\) is written \(x_i=(x_{i,0},\ldots,x_{i,m_i})\).  Eliminating \(x_{i,0}\) through
\[
x_{i,0}=1-\sum_{j=1}^{k_i}x_{i,j}
\]
leaves \(k_i\) independent variables for player \(i\).

At a totally mixed Nash equilibrium, every player is indifferent among all pure strategies.  Relative to strategy \(0\), the indifference equations are
\begin{equation}
\label{eq:indifference}
f_{i,j}(\mathbf{x}) = u_i(j, \mathbf{x}_{-i}) - u_i(0, \mathbf{x}_{-i}) = 0, \quad 
\forall i \in \mathcal{I}, \forall j \in \{1, \dots, k_i\}.
\end{equation}
Here \(u_i(s,\mathbf x_{-i})\) is the expected payoff from pure strategy \(s\) against the mixed profile \(\mathbf x_{-i}\).

A real solution of \eqref{eq:indifference} is a totally mixed Nash equilibrium only if it lies in the relative interior of every player's probability simplex, namely
\[
x_{i,j}>0\quad (1\le j\le k_i),
\qquad
x_{i,0}:=1-\sum_{j=1}^{k_i}x_{i,j}>0
\quad (1\le i\le N).
\]
For a totally mixed profile the indifference equations make every pure strategy a best response, so these strict simplex inequalities are the remaining feasibility conditions.  A complex solution, or even a real solution outside these simplices, is not a Nash equilibrium.

The expected payoff differences are multilinear because players randomize independently: every monomial contains at most one probability variable from each player.  In particular, powers such as $x_{a,1}^2$ do not occur.

\subsection{Generic coefficient systems from payoff tensors}
\label{sec:generic-realization}

For player $i$, let
\[
\mathcal A_i=\prod_{q\in B_i}S_q
\]
be the set of pure action profiles of the other players in the local payoff table, and define the block-multiaffine polynomial space
\[
\mathcal P_i
=\bigotimes_{q\in B_i}
\operatorname{span}_{\mathbb C}\{1,x_{q,1},\ldots,x_{q,k_q}\}.
\]
The tensor product is interpreted as $\mathbb C$ when $B_i$ is empty.

\begin{proposition}
\label{prop:generic-realization}
After eliminating the baseline probabilities, the linear map from the complexified local payoff tables to the coefficient tuple
\[
\bigl(f_{i,s}\bigr)_{1\le i\le N,\ 1\le s\le k_i}
\in
\prod_{i=1}^N\mathcal P_i^{k_i}
\]
is surjective.  Consequently, every nonempty Zariski-open set of coefficient systems with the prescribed network supports has a nonempty Zariski-open preimage in payoff space.
\end{proposition}

\begin{proof}
For $q\in B_i$, the affine forms
\[
x_{q,0}=1-\sum_{t=1}^{k_q}x_{q,t},
\qquad x_{q,1},\ldots,x_{q,k_q},
\]
form a basis of $\operatorname{span}_{\mathbb C}\{1,x_{q,1},\ldots,x_{q,k_q}\}$.  Their products over $q\in B_i$ therefore form a basis of $\mathcal P_i$.

Fix arbitrary coefficients $c_{i,s,\alpha}\in\mathbb C$, where $\alpha\in\mathcal A_i$.  Choose the local pure payoffs so that
\[
u_i(0,\alpha)=0,
\qquad
u_i(s,\alpha)=c_{i,s,\alpha}
\quad(1\le s\le k_i).
\]
Then the expected payoff difference is
\[
f_{i,s}(x)
=\sum_{\alpha\in\mathcal A_i}
c_{i,s,\alpha}
\prod_{q\in B_i}x_{q,\alpha_q}.
\]
Thus every element of every $\mathcal P_i^{k_i}$ is realized, independently across players and equations.  The payoff-to-coefficient map is therefore onto.  The inverse image of a nonempty Zariski-open set under a surjective linear map is again nonempty and Zariski open, which proves the final assertion.  Since the map and the BKK nondegeneracy conditions are defined over $\mathbb R$, this preimage also contains a Euclidean open dense subset of the real payoff space.
\end{proof}

\section{The Classical Permanent Formula}
\label{sec:datta}

Let
\[
\Omega=\{(i,s):1\le i\le N,\ 1\le s\le k_i\},
\qquad d=|\Omega|=\sum_{i=1}^N k_i.
\]
Both the equations and the independent variables are indexed by \(\Omega\).

\subsection{The Polynomial Graph}
\label{sec:polynomial_graph}

\begin{definition}
\label{def:structure_matrix}
The \emph{polynomial matrix} \(A_{\mathrm{poly}}\in\{0,1\}^{d\times d}\) is the equation--variable incidence matrix
\[
 (A_{\mathrm{poly}})_{(i,s),(j,t)}=
\begin{cases} 
 1,&\text{if }x_{j,t}\text{ appears in }f_{i,s},\\
 0 & \text{otherwise.}
\end{cases}
\]
After identifying the row and column label sets with \(\Omega\), this matrix is the adjacency matrix of the directed \emph{polynomial graph} \(G_{\mathrm{poly}}\): an arc \((i,s)\to(j,t)\) records the corresponding incidence.
\end{definition}

For a \(d\times d\) matrix \(B\), its permanent is
\[
\perm(B)=\sum_{\sigma\in S_d}\prod_{r\in\Omega}B_{r,\sigma(r)}.
\]
For a zero--one adjacency matrix, every nonzero term is a directed cycle cover: a spanning collection of vertex-disjoint directed cycles, equivalently a permutation \(\sigma\) for which every selected arc \(r\to\sigma(r)\) exists.

Define the column-scaled matrix
\[
M_{(i,s),(j,t)}=
\frac{(A_{\mathrm{poly}})_{(i,s),(j,t)}}{(k_j!)^{1/k_j}}.
\]
Every permanent term uses every column exactly once.  Since player \(j\) contributes \(k_j\) columns, column scaling gives the identity
\begin{equation}
\label{eq:weighted-permanent}
\perm(M)=\frac{\perm(A_{\mathrm{poly}})}{\prod_{j=1}^N k_j!}.
\end{equation}

For full normal-form support, the permanent count is equivalent to the constrained-partition formula of McKelvey and McLennan \cite{McKelveyMcLennan1997}.  Datta's formulation incorporates prescribed structural zeros and the corresponding sparse polynomial systems \cite{Datta2003,DattaPolynomialGraphs2006,Datta2010}.  In the present notation the result is as follows.
\begin{theorem}
\label{thm:Datta}
For payoff coefficients in a Zariski-open dense subset of the coefficient space with the prescribed supports, the number \(\mathcal D\) of isolated solutions of \eqref{eq:indifference} in \((\mathbb C^*)^d\), counted with multiplicity, is
\begin{equation}
\label{eq:datta_formula}
\mathcal D
=\frac{\perm(A_{\mathrm{poly}})}{\prod_{j=1}^N k_j!}
=\perm(M).
\end{equation}
\end{theorem}

\begin{corollary}
\label{cor:cycle_covers}
The number of directed cycle covers of \(G_{\mathrm{poly}}\) is \(\perm(A_{\mathrm{poly}})\).  Consequently, \(\mathcal D\) is this labeled cycle-cover count divided by \(\prod_jk_j!\).
\end{corollary}

The normalization of \(M\) in \eqref{eq:weighted-permanent} is induced by column scaling.  In the mixed-volume calculation below, the standard simplex \(\Delta_k=\conv(0,e_1,\ldots,e_k)\) has Euclidean volume \(1/k!\), while its normalized lattice volume is \(1\).

\section{The Player-Level Degree Identity}
\label{sec:mixed-volume}

\subsection{From mixed volume to balanced assignments}

The block expansion of the mixed volume underlying \cref{thm:Datta} is derived below.

For each player \(j\), let \(V_j\cong\mathbb R^{k_j}\) be the coordinate subspace spanned by that player's independent variables, and let
\[
\Delta_{k_j}=\conv(0,e_{j,1},\ldots,e_{j,k_j})\subset V_j.
\]
Thus \(\mathbb R^d=\bigoplus_{j=1}^NV_j\).  We use the BKK convention in which
\(MV(Q_1,\ldots,Q_d)\) is the coefficient of \(\lambda_1\cdots\lambda_d\) in
\(\operatorname{vol}_d(\lambda_1Q_1+\cdots+\lambda_dQ_d)\).  With this convention,
\[
MV(\underbrace{\Delta_k,\ldots,\Delta_k}_{k\text{ copies}})=1.
\]

\begin{proposition}
\label{prop:network-polytope}
For the prescribed generic support, the Newton polytope of the equation \(f_{i,s}\) is
\begin{equation}
\label{eq:newton-minkowski}
P_{i,s}=\sum_{j\in\mathcal N_i\setminus\{i\}}\Delta_{k_j},
\end{equation}
where the summands lie in their respective coordinate subspaces \(V_j\).
\end{proposition}

\begin{proof}
After substituting \(x_{j,0}=1-\sum_{t=1}^{k_j}x_{j,t}\), a generic multilinear payoff difference has, for each neighboring player \(j\), exponents in \(\{0,e_{j,1},\ldots,e_{j,k_j}\}\).  Its exponent set is the sum of these blockwise exponent sets.  Taking convex hulls and using
\(\conv(S_1+\cdots+S_m)=\conv(S_1)+\cdots+\conv(S_m)\) gives \eqref{eq:newton-minkowski}.
\end{proof}

Recall that, for each player \(i\),
\[
B_i(G)=\mathcal N_i\setminus\{i\}
\qquad\text{and}\qquad
L_i^G(z)=\sum_{j\in B_i(G)}z_j.
\]
Thus \(B_i(G)\) is the set of player-variable blocks that may occur in an equation of player \(i\); an empty sum is understood as zero.  For a row \(r=(i,s)\in\Omega\), an \emph{admissible block assignment} is a map
\[
\phi:\Omega\longrightarrow\{1,\ldots,N\},
\qquad \phi(i,s)\in B_i(G).
\]
It is \emph{balanced} if \(|\phi^{-1}(j)|=k_j\) for every player \(j\).

For a polynomial \(F(z_1,\ldots,z_N)\), we write \([z^{\mathbf k}]F\) for the coefficient of \(z_1^{k_1}\cdots z_N^{k_N}\).

\begin{proposition}
\label{prop:balanced-assignments}
The BKK mixed volume of the indifference system is the number of admissible balanced block assignments.  More precisely,
\begin{equation}
\label{eq:main-permanent}
\begin{aligned}
\mathcal D(G,\mathbf k)
&=MV\bigl((P_r)_{r\in\Omega}\bigr)\\
&=[z^{\mathbf k}]\prod_{i=1}^N\bigl(L_i^G(z)\bigr)^{k_i}\\
&=\#\{\text{admissible balanced block assignments}\}\\
&=\frac{\perm(A_{\mathrm{poly}})}{\prod_{j=1}^Nk_j!}
=\perm(M).
\end{aligned}
\end{equation}
\end{proposition}

\begin{proof}
Expand \(\prod_i(L_i^G)^{k_i}\) by viewing its \(k_i\) copies of \(L_i^G\) as the labeled equation rows \((i,1),\ldots,(i,k_i)\).  Choosing \(z_j\) from the factor associated with row \((i,s)\) is the same as assigning that row to block \(j\in B_i(G)\).  The exponent of \(z_j\) records the number of rows assigned to block \(j\).  Hence \([z^{\mathbf k}]\prod_i(L_i^G)^{k_i}\) counts exactly the admissible balanced assignments.

Multilinearity of mixed volume and \eqref{eq:newton-minkowski} give
\[
MV\bigl((P_r)_{r\in\Omega}\bigr)
=\sum_{\substack{\phi:\Omega\to\{1,\ldots,N\}\\
\phi(i,s)\in B_i(G)\ \forall(i,s)\in\Omega}}
MV\bigl((\Delta_{k_{\phi(r)}})_{r\in\Omega}\bigr).
\]
Fix \(\phi\) and put \(R_j=\phi^{-1}(j)\).  Because the spaces \(V_j\) are complementary, the relevant volume polynomial factors as
\[
\operatorname{vol}_d\!\left(\sum_{r\in\Omega}\lambda_r\Delta_{k_{\phi(r)}}\right)
=\prod_{j=1}^N
\frac{\left(\sum_{r\in R_j}\lambda_r\right)^{k_j}}{k_j!}.
\]
The coefficient of \(\prod_{r\in\Omega}\lambda_r\) is zero unless \(|R_j|=k_j\) for every \(j\).  Under these balance conditions, the multinomial coefficient \(k_j!\) cancels the Euclidean simplex factor \(1/k_j!\) in every block, so the coefficient is one.  This proves the first equality in \eqref{eq:main-permanent}.

To obtain the permanent, forget the second label in a variable column: a permutation selected by \(A_{\mathrm{poly}}\) sends every row \(r\) to a column \((j,t)\) and hence induces \(\phi(r)=j\).  Because all columns are used once, this assignment is balanced.  Conversely, for a fixed balanced assignment, the \(k_j\) rows assigned to player \(j\) can be bijected with the labeled columns \((j,1),\ldots,(j,k_j)\) in exactly \(k_j!\) ways.  Thus every balanced assignment has \(\prod_jk_j!\) labeled refinements, proving the remaining equalities.
\end{proof}

\begin{remark}
\label{rem:intersection-theoretic-interface}
Let $X_{\mathbf k}=\prod_j\mathbb P^{k_j}$ and let $H_j$ be the hyperplane class of its $j$th factor.  After multihomogenization, the payoff-difference system is a section of the rank-$d$ split bundle
\[
\mathcal E_G
=\bigoplus_{i=1}^N
\mathcal O_{X_{\mathbf k}}\!\left(\sum_{j\in B_i(G)}H_j\right)^{\oplus k_i}.
\]
On the generic locus where its zero scheme is finite and avoids the coordinate boundary, its degree is the top Chern number
\[
\int_{X_{\mathbf k}}c_d(\mathcal E_G)
=[H_1^{k_1}\cdots H_N^{k_N}]
\prod_{i=1}^N\left(\sum_{j\in B_i(G)}H_j\right)^{k_i}.
\]
In $A^*(X_{\mathbf k})\otimes_{\mathbb Z}\mathbb Z[x]$, the edge-marked refinement is
\[
\Phi_{\mathbf k}(x)
=\int_{X_{\mathbf k}}
\prod_{i=1}^N\left(\sum_{j\ne i}x_{ji}H_j\right)^{k_i}.
\]
This connects the support formulation with the vector-bundle approach \cite{AboPortakalSodomaco2026}.  It is a global intersection-number identity; its summands are not identified with local multiplicities of individual mixed cells or tropical intersection points.
\end{remark}

\subsection{The network degree polynomial and balanced factors}
\label{sec:degree-polynomial}

Let
\[
\mathcal E_N=\{(j,i):1\le i,j\le N,\ j\ne i\}
\]
be the set of potential dependency arcs, where $(j,i)$ denotes $j\to i$.  Introduce variables $x=(x_{ji})_{(j,i)\in\mathcal E_N}$ and define
\begin{equation}
\label{eq:degree-polynomial}
\Phi_{\mathbf k}(x)
=[z^{\mathbf k}]
\prod_{i=1}^N\left(\sum_{j\ne i}x_{ji}z_j\right)^{k_i}.
\end{equation}
For an arc set $E\subseteq\mathcal E_N$, let $\mathbf 1_E$ be its zero--one incidence vector.

\begin{proposition}
\label{prop:degree-specialization}
For every directed dependency graph $G=([N],E)$,
\[
\mathcal D(G,\mathbf k)=\Phi_{\mathbf k}(\mathbf 1_E).
\]
\end{proposition}

\begin{proof}
After the specialization $x=\mathbf 1_E$, the linear form in the $i$th factor of \eqref{eq:degree-polynomial} becomes $L_i^G(z)$.  The claim follows from \eqref{eq:main-permanent}.
\end{proof}

A \emph{balanced $\mathbf k$-factor} supported on $E$ is a family
\[
a=(a_{ji})_{(j,i)\in\mathcal E_N}\in\mathbb Z_{\ge0}^{\mathcal E_N}
\]
such that $a_{ji}=0$ for $(j,i)\notin E$ and
\begin{equation}
\label{eq:eulerian-margins}
\sum_{j\ne i}a_{ji}=k_i\quad(1\le i\le N),
\qquad
\sum_{i\ne j}a_{ji}=k_j\quad(1\le j\le N).
\end{equation}
Thus vertex $i$ has both in-degree and out-degree $k_i$ in the directed multigraph with arc multiplicities $a_{ji}$.  No connectivity condition is imposed.  Denote the set of these factors by $\mathfrak E_{\mathbf k}(E)$ and assign
\begin{equation}
\label{eq:eulerian-weight}
w_{\mathbf k}(a)
=\prod_{i=1}^N
\frac{k_i!}{\prod_{j\ne i}a_{ji}!}.
\end{equation}

\begin{theorem}
\label{thm:eulerian-factor}
The network degree polynomial is the generating polynomial of balanced $\mathbf k$-factors:
\begin{equation}
\label{eq:eulerian-generating}
\Phi_{\mathbf k}(x)
=\sum_{a\in\mathfrak E_{\mathbf k}(\mathcal E_N)}
w_{\mathbf k}(a)
\prod_{(j,i)\in\mathcal E_N}x_{ji}^{a_{ji}}.
\end{equation}
Consequently,
\begin{equation}
\label{eq:eulerian-degree}
\mathcal D(G,\mathbf k)
=\sum_{a\in\mathfrak E_{\mathbf k}(E(G))}w_{\mathbf k}(a).
\end{equation}
\end{theorem}

\begin{proof}
For each player $i$, the multinomial theorem gives
\[
\left(\sum_{j\ne i}x_{ji}z_j\right)^{k_i}
=\sum_{\substack{a_{ji}\ge0\ (j\ne i)\\
\sum_{j\ne i}a_{ji}=k_i}}
\frac{k_i!}{\prod_{j\ne i}a_{ji}!}
\prod_{j\ne i}(x_{ji}z_j)^{a_{ji}}.
\]
The exponent of $z_j$ in the product over $i$ is $\sum_{i\ne j}a_{ji}$.  Extracting $z^{\mathbf k}$ therefore imposes the second family of equalities in \eqref{eq:eulerian-margins}; the first family is imposed within each multinomial expansion.  The remaining coefficient and monomial are exactly \eqref{eq:eulerian-weight} and \eqref{eq:eulerian-generating}.  Specializing to $\mathbf 1_{E(G)}$ removes precisely the factors using an arc outside $G$, which proves \eqref{eq:eulerian-degree}.
\end{proof}

The edge variables retain support information that is lost in the expanded zero--one permanent; the remaining results extract its computational, transportation, and scaling consequences.

\begin{proposition}
\label{prop:player-dp}
Put $K(\mathbf k)=\prod_{j=1}^N(k_j+1)$ and $d=\sum_i k_i$.  For a fixed support $E$, the coefficient $\mathcal D_{\mathbf k}(E)$ can be computed with
\[
O\!\left(K(\mathbf k)\sum_{i=1}^Nk_i|B_i|\right)
\]
arithmetic operations and $O(K(\mathbf k))$ memory.  A coarser bound for the operation count is
\[
O(NdK(\mathbf k)).
\]
\end{proposition}

\begin{proof}
Multiply the linear forms $L_i^E$ one at a time, using $k_i$ copies of the $i$th form, and discard every monomial whose exponent in coordinate $j$ exceeds $k_j$.  There are at most $K(\mathbf k)$ retained exponent vectors, and multiplication by one copy of $L_i^E$ makes at most $|B_i|$ updates per retained vector.  Two coefficient arrays suffice for the computation.
\end{proof}

This dynamic program uses the player-level state space.  Ryser's formula on the expanded $d\times d$ incidence matrix uses $O(d2^d)$ arithmetic operations.  When $k_i=k$ for all $i$, the corresponding state counts are $(k+1)^N$ and $2^{Nk}$, respectively; which representation is preferable depends on the strategy format and sparsity.

We next record the resulting boundary between exact counting and fixed-player tractability.  To make the input model explicit, the positive integers $k_i$ are written in unary, so the expanded equation index set $\Omega$ has size polynomial in the input length.

\begin{theorem}
\label{thm:degree-complexity}
Given a loopless directed graph $G=([N],E)$ and a positive strategy vector $\mathbf k$, exact evaluation of $\mathcal D(G,\mathbf k)$ is $\#\mathrm P$-complete under polynomial-time Turing reductions, even when $\mathbf k=\mathbf 1$.  On the other hand, for every fixed $N$ it is computable in deterministic polynomial time in $d=\sum_i k_i$.  More precisely, the algorithm in \cref{prop:player-dp} uses at most
\begin{equation}
\label{eq:fixed-player-complexity}
O\!\left((N-1)d\left(1+\frac dN\right)^N\right)
\end{equation}
arithmetic operations and stores at most $\left(1+d/N\right)^N$ integers of $O(d\log N)$ bits.
\end{theorem}

\begin{proof}
By \cref{prop:balanced-assignments}, $\mathcal D(G,\mathbf k)$ is the number of admissible balanced block assignments of the $d$ labeled equation rows.  Under the stated unary encoding, such an assignment is a polynomial-size certificate, and its support and block-cardinality conditions can be checked in polynomial time.  Hence exact evaluation belongs to $\#\mathrm P$.

For hardness, let $A\in\{0,1\}^{n\times n}$ be arbitrary and form
\[
B=
\begin{pmatrix}
0&A\\
I_n&0
\end{pmatrix}.
\]
The matrix $B$ has zero diagonal and is therefore the receiver--row adjacency matrix of a loopless directed graph $G_B$ on $2n$ vertices.  Every nonzero permanent term of $B$ maps the first $n$ rows to the last $n$ columns through $A$, while the last $n$ rows are forced through $I_n$ to the first $n$ columns.  This gives a bijection between the permanent terms of $A$ and those of $B$, and hence
\[
\mathcal D(G_B,\mathbf 1)=\operatorname{perm}(B)=\operatorname{perm}(A).
\]
Since exact evaluation of the permanent of a zero--one matrix is $\#\mathrm P$-complete \cite{Valiant1979}, this count-preserving reduction proves hardness.

For the fixed-player assertion, the arithmetic bound in \cref{prop:player-dp}, the estimate $|B_i|\le N-1$, and the arithmetic--geometric mean inequality give
\[
\sum_i k_i|B_i|\le (N-1)d,
\qquad
K(\mathbf k)=\prod_i(k_i+1)
\le\left(\frac{d+N}{N}\right)^N.
\]
This proves \eqref{eq:fixed-player-complexity}.  At every intermediate stage, all coefficients are nonnegative and their sum is at most $N^d$, so each stored integer has $O(d\log N)$ bits.  Consequently, for fixed $N$ the same algorithm also runs in polynomial bit complexity.
\end{proof}

The theorem does not assert fixed-parameter tractability with parameter $N$: the exponent in \eqref{eq:fixed-player-complexity} depends on $N$.  If the strategy dimensions are supplied only in binary as a succinct input, the bound is pseudopolynomial in those encoded values.  Nor does the state-space inequality $K(\mathbf k)\le2^d$ imply uniform runtime dominance over permanent algorithms, since the polynomial update factors differ.  The gain is strongest when the player blocks have nontrivial strategy dimension; for $k_i=k$, the two state counts are $(k+1)^N$ and $(2^k)^N$.

\begin{corollary}
\label{cor:reversal-symmetry}
Let $E^{\mathsf T}=\{i\to j:j\to i\in E\}$ be the reversed dependency support.  Then
\[
\mathcal D_{\mathbf k}(E)=\mathcal D_{\mathbf k}(E^{\mathsf T}).
\]
\end{corollary}

\begin{proof}
Transposition sends a balanced factor $a$ on $E$ to a balanced factor $a^{\mathsf T}$ on $E^{\mathsf T}$.  It preserves the margins and the product of all factorial denominators in \eqref{eq:eulerian-weight}, and hence preserves the weight.
\end{proof}

Every player in a positive-degree graph lies on a directed cycle: a supported balanced factor has positive in- and out-degree at each vertex and decomposes into directed cycles.

When $k_i=1$ for every player, each nonzero $a_{ji}$ equals one, every weight in \eqref{eq:eulerian-weight} is one, and the balanced factors are the directed cycle covers of the player graph.  Thus \eqref{eq:eulerian-degree} recovers the usual permanent interpretation in the binary-strategy case.

For example, let $G$ be the complete loopless digraph on three vertices and let $\mathbf k=(2,2,2)$.  Its balanced factors are
\[
A(a)=
\begin{pmatrix}
0&a&2-a\\
2-a&0&a\\
a&2-a&0
\end{pmatrix},
\qquad a\in\{0,1,2\}.
\]
Their weights are $1,8,1$, respectively, and hence $\mathcal D(G,(2,2,2))=10$.

\subsection{Nonvanishing: a capacitated Hall criterion}
\label{sec:hall}

For $I\subseteq[N]$, define its set of predecessor blocks by
\[
N_G^-(I)=\{j\in[N]:j\to i\in E(G)\text{ for some }i\in I\}.
\]

At the expanded equation--variable level, Datta already observed that nonvanishing is equivalent to a perfect matching in the bipartite polynomial graph \cite[Corollary~2]{DattaPolynomialGraphs2006}.  The following theorem is the player-block compression of that criterion: repeated strategy labels are aggregated into the capacities $k_i$, yielding Hall inequalities indexed by player sets and the compact flow test in \cref{cor:flow-nonvanishing}.

\begin{theorem}
\label{thm:hall-nonvanishing}
The generic algebraic degree is positive if and only if
\begin{equation}
\label{eq:capacitated-hall}
\sum_{i\in I}k_i
\le
\sum_{j\in N_G^-(I)}k_j
\qquad\text{for every }I\subseteq[N].
\end{equation}
\end{theorem}

\begin{proof}
Form a bipartite graph with the equation labels $(i,s)\in\Omega$ on the left and the variable labels $(j,t)\in\Omega$ on the right.  Join $(i,s)$ to $(j,t)$ exactly when $j\to i\in E(G)$.  A perfect matching is equivalent to a balanced block assignment.  By \cref{prop:balanced-assignments}, it exists exactly when $\mathcal D(G,\mathbf k)>0$.

The left vertices in a collection $I$ of player blocks have cardinality $\sum_{i\in I}k_i$ and exactly $\sum_{j\in N_G^-(I)}k_j$ neighbors.  Hall's condition therefore implies \eqref{eq:capacitated-hall}.  Conversely, take any set $S$ of left vertices and let $I$ be the set of players represented in $S$.  Then its neighborhood is the union of the right blocks indexed by $N_G^-(I)$, so
\[
|S|\le\sum_{i\in I}k_i
\le\sum_{j\in N_G^-(I)}k_j
=|N(S)|.
\]
Thus \eqref{eq:capacitated-hall} implies Hall's condition for every $S$, and the bipartite graph has a perfect matching.
\end{proof}

By reversal symmetry, \eqref{eq:capacitated-hall} is equivalent to the analogous inequalities with the successor sets $N_G^+(I)=\{j:i\to j\text{ for some }i\in I\}$.  For the complete loopless graph, only the singleton inequalities are nontrivial, and positivity reduces to
\[
k_i\le\sum_{j\ne i}k_j\qquad(1\le i\le N).
\]
This recovers the classical balanced-format condition \cite{McKelveyMcLennan1997}, called the within-boundary-format condition in the vector-bundle formulation \cite{AboPortakalSodomaco2026}.  Boundary format itself refers to the equality case $k_i=\sum_{j\ne i}k_j$ for some $i$.

\begin{corollary}
\label{cor:flow-nonvanishing}
Whether $\mathcal D(G,\mathbf k)$ is positive can be decided by a capacitated maximum-flow computation on a network with $2N+2$ vertices and $|E(G)|+2N$ arcs.
\end{corollary}

\begin{proof}
Create vertices $i_L$ and $j_R$ for the equation and variable blocks.  Add an arc from the source to $i_L$ of capacity $k_i$, an arc from $i_L$ to $j_R$ of capacity $d$ whenever $j\to i\in E(G)$, and an arc from $j_R$ to the sink of capacity $k_j$.  The total demand is $d=\sum_i k_i$.  By integrality of maximum flow, a flow of value $d$ is equivalent to a balanced block assignment and hence, by \cref{thm:hall-nonvanishing}, to positive degree.
\end{proof}

Thus the nonvanishing decision problem is polynomial-time solvable even though exact degree evaluation is $\#\mathrm P$-complete by \cref{thm:degree-complexity}.

\subsection{Minimal positive supports: transportation forests}
\label{sec:minimal-supports}

For an arc set $E\subseteq\mathcal E_N$, recall that $\mathcal D_{\mathbf k}(E)$ denotes the degree of the graph $([N],E)$ and consider the transportation polytope with structural zeros
\begin{equation}
\label{eq:transportation-polytope}
\mathcal T_E(\mathbf k)
=\left\{a\in\mathbb R_{\ge0}^{E}:
\sum_{j:j\to i\in E}a_{ji}=k_i,\quad
\sum_{i:j\to i\in E}a_{ji}=k_j
\ \text{for all }i,j\right\}.
\end{equation}
Its integer points are precisely the balanced $\mathbf k$-factors supported on $E$.  The integrality of the corresponding flow problem also shows that $\mathcal T_E(\mathbf k)$ is nonempty exactly when $\mathcal D_{\mathbf k}(E)>0$.

Let $\mathcal B(E)$ be the bipartite graph with left vertices $i_L$, right vertices $j_R$, and an edge $i_Lj_R$ for every dependency arc $j\to i\in E$.  The classical support criterion for transportation polytopes identifies their vertices with feasible tables whose positive support is a spanning forest \cite{KleeWitzgall1968,DeLoeraKim2014}.  In the present setting it gives the following characterization of support-minimal feasible network architectures.

\Needspace{10\baselineskip}
\begin{theorem}
\label{thm:minimal-positive-support}
For $E\subseteq\mathcal E_N$, the following conditions are equivalent.
\begin{enumerate}[label=\textup{(\roman*)}]
\item $E$ is inclusion-minimal among the arc sets of positive degree; that is,
\[
\mathcal D_{\mathbf k}(E)>0,
\qquad
\mathcal D_{\mathbf k}(E\setminus\{e\})=0
\quad\text{for every }e\in E.
\]
\item The bipartite graph $\mathcal B(E)$ is a forest and $\mathcal T_E(\mathbf k)$ contains a point that is positive on every edge of $E$.
\item The graph $\mathcal B(E)$ is a forest, every connected component $C$ satisfies
\begin{equation}
\label{eq:component-balance}
\sum_{i_L\in C}k_i=\sum_{j_R\in C}k_j,
\end{equation}
and the following cut inequalities hold.  For $e=i_Lj_R\in E$, let $U_e$ be the component of $\mathcal B(E)-e$ containing $i_L$.  Then
\begin{equation}
\label{eq:forest-cut}
\delta_e
:=\sum_{p_L\in U_e}k_p-\sum_{q_R\in U_e}k_q
>0.
\end{equation}
\end{enumerate}
When these conditions hold, the transportation polytope consists of a single integral point $a^E$, given explicitly by
\begin{equation}
\label{eq:forest-flow}
a^E_{ji}=\delta_{i_Lj_R}.
\end{equation}
Moreover,
\begin{equation}
\label{eq:minimal-support-degree}
\mathcal D_{\mathbf k}(E)
=\prod_{i=1}^N
\frac{k_i!}{\prod_{j:j\to i\in E}(a^E_{ji})!}.
\end{equation}
\end{theorem}

\begin{proof}
Assume \textup{(i)} and choose an integral point $a\in\mathcal T_E(\mathbf k)$.  Minimality forces $a_e>0$ for every $e\in E$.  If $\mathcal B(E)$ contained a cycle, add and subtract the same positive integer alternately around that cycle.  Choosing the subtracted amount as the minimum entry on the negative half of the cycle preserves all margins, keeps every entry nonnegative, and makes at least one edge zero.  This would give a factor supported on a proper subset of $E$, a contradiction.  Hence $\mathcal B(E)$ is a forest and \textup{(ii)} follows.

Suppose \textup{(ii)} holds.  Summing the left and right margin equations over a connected component gives \eqref{eq:component-balance}.  If $e=i_Lj_R$ is deleted and $U_e$ contains $i_L$, all internal edge contributions cancel when the right margins in $U_e$ are subtracted from the left margins there.  The only remaining contribution is $a_e$, and therefore
\[
a_e=\sum_{p_L\in U_e}k_p-\sum_{q_R\in U_e}k_q=\delta_e>0.
\]
This proves \textup{(iii)} and \eqref{eq:forest-flow}.

Finally assume \textup{(iii)}.  On each tree component, orient every edge from its left endpoint to its right endpoint.  The signed incidence matrix has rank one less than the number of vertices.  Condition \eqref{eq:component-balance} is exactly the compatibility condition for the margin equations, and the cut computation above gives their unique solution.  By \eqref{eq:forest-cut}, that solution is positive; by \eqref{eq:forest-flow}, it is integral.  Thus $\mathcal T_E(\mathbf k)=\{a^E\}$ and the degree is positive.  Deleting any edge destroys feasibility, because a feasible point after deletion would be a second point of $\mathcal T_E(\mathbf k)$.  This proves \textup{(i)}.  Formula \eqref{eq:minimal-support-degree} now follows from \cref{thm:eulerian-factor}, since $a^E$ is the only balanced factor supported on $E$.
\end{proof}

\begin{corollary}
\label{cor:minimal-support-size}
\label{cor:positive-support-sparsification}
If $E$ is inclusion-minimal of positive degree and $\mathcal B(E)$ has $c$ connected components, then
\[
|E|=2N-c\le2N-1.
\]
Equality holds if and only if $\mathcal B(E)$ is connected.  As a bound uniform over strategy vectors, $2N-1$ is sharp for every $N\ge3$.
Moreover, every positive-degree dependency graph contains a positive-degree spanning subgraph with at most $2N-1$ arcs.
\end{corollary}

\begin{proof}
Every component is a tree, so summing $|V(C)|-1$ over the components gives $|E|=2N-c$.  Hence equality in the bound is equivalent to connectedness.

For sharpness, take $N\ge3$ and
\[
E_N^\star
=\{1\to i,\ i\to1:2\le i\le N\}\cup\{2\to3\}.
\]
Its bipartite support graph is obtained by joining the two tree components of the bidirected star with the edge corresponding to $2\to3$, and is therefore a tree with $2N-1$ edges.  Give the arcs the positive integral multiplicities
\[
\begin{gathered}
a_{2,1}=1,\quad a_{3,1}=2,\quad a_{i,1}=1\ (i\ge4),\\
a_{1,2}=2,\quad a_{1,3}=1,\quad a_{1,i}=1\ (i\ge4),
\quad a_{2,3}=1.
\end{gathered}
\]
The resulting in- and out-degree vector is $\mathbf k=(N,2,2,1,\ldots,1)$.  Thus every edge of the tree carries positive flow, so \cref{thm:minimal-positive-support} shows that $E_N^\star$ is inclusion-minimal of positive degree and attains the bound.

The construction varies $\mathbf k$ with $N$; for $\mathbf k=\mathbf1$, every minimal positive support is a directed cycle cover and has $N$ arcs.  Finally, any positive support contains a minimal one by successive deletion, proving the last assertion.  Rooting each tree and accumulating the signed vertex weights computes all cut values, and hence \eqref{eq:minimal-support-degree}, in $O(N+|E|)$ arithmetic operations once the required factorials are tabulated.
\end{proof}

\section{Support Comparative Statics}
\label{sec:monotonicity}

Consider dependency graphs on the same player set and with the same strategy vector \(\mathbf k=(k_1,\ldots,k_N)\).  Coefficients are generic within the prescribed support associated with each graph.  We apply standard M\"obius inversion on the Boolean lattice \cite{Rota1964} to the exact-support decomposition supplied by the player-level degree identity.  The resulting increasing-differences statements belong to the usual supermodular framework \cite{Topkis1998}; the network-specific information lies in the balanced-factor coefficients and in the residual feasibility criteria below.

For $a\in\mathfrak E_{\mathbf k}(\mathcal E_N)$, let
\[
\operatorname{supp}(a)=\{(j,i)\in\mathcal E_N:a_{ji}>0\}.
\]
For $F\subseteq\mathcal E_N$, define the exact-support coefficient
\begin{equation}
\label{eq:mobius-coefficient}
\mu_{\mathbf k}(F)
=\sum_{\substack{a\in\mathfrak E_{\mathbf k}(\mathcal E_N)\\
\operatorname{supp}(a)=F}}
w_{\mathbf k}(a).
\end{equation}

\begin{proposition}
\label{prop:mobius-transportation-dimension}
Suppose $\mu_{\mathbf k}(F)>0$, and let $c(F)$ be the number of connected components of $\mathcal B(F)$.  Then
\begin{equation}
\label{eq:transportation-dimension}
\dim\mathcal T_F(\mathbf k)=|F|-2N+c(F).
\end{equation}
Moreover, $F$ is inclusion-minimal of positive degree if and only if $\mu_{\mathbf k}(F)>0$ and the dimension in \eqref{eq:transportation-dimension} is zero.
\end{proposition}

\begin{proof}
The condition $\mu_{\mathbf k}(F)>0$ means that $\mathcal T_F(\mathbf k)$ contains an integral point positive on every coordinate indexed by $F$.  Hence none of the nonnegativity constraints reduces its affine dimension.  The signed incidence matrix of $\mathcal B(F)$ has rank $2N-c(F)$, which proves \eqref{eq:transportation-dimension}.  A zero-dimensional polytope with a full-support point is a singleton and loses feasibility when any edge is deleted.  The converse follows from \cref{thm:minimal-positive-support}, since a minimal positive support is a forest carrying one positive integral flow.
\end{proof}

For disjoint arc sets $E,Q\subseteq\mathcal E_N$, set
\[
\Delta_Q\mathcal D_{\mathbf k}(E)
=\sum_{R\subseteq Q}(-1)^{|Q|-|R|}
\mathcal D_{\mathbf k}(E\cup R).
\]

\begin{proposition}
\label{thm:mobius-differences}
For every $E\subseteq\mathcal E_N$,
\begin{equation}
\label{eq:mobius-expansion}
\mathcal D_{\mathbf k}(E)
=\sum_{F\subseteq E}\mu_{\mathbf k}(F).
\end{equation}
Equivalently, the M\"obius transform of the degree function is nonnegative:
\begin{equation}
\label{eq:mobius-inversion}
\mu_{\mathbf k}(F)
=\sum_{E\subseteq F}(-1)^{|F|-|E|}
\mathcal D_{\mathbf k}(E)
\ge0.
\end{equation}
More generally, if $E\cap Q=\varnothing$, then
\begin{equation}
\label{eq:higher-difference}
\Delta_Q\mathcal D_{\mathbf k}(E)
=\sum_{\substack{F\subseteq\mathcal E_N\\F\setminus E=Q}}
\mu_{\mathbf k}(F)
\ge0.
\end{equation}
\end{proposition}

\begin{proof}
Equation \eqref{eq:eulerian-degree} expresses $\mathcal D_{\mathbf k}(E)$ as the sum of the weights of the balanced factors whose supports are contained in $E$.  Grouping them by exact support gives \eqref{eq:mobius-expansion}, and Boolean-lattice M\"obius inversion gives \eqref{eq:mobius-inversion}.

Substitute \eqref{eq:mobius-expansion} into the definition of $\Delta_Q$.  For a fixed exact support $F$, its total coefficient is
\[
\sum_{\substack{R\subseteq Q\\F\subseteq E\cup R}}
(-1)^{|Q|-|R|}.
\]
This sum is one when $F\setminus E=Q$ and zero otherwise.  Summing the surviving nonnegative weights proves \eqref{eq:higher-difference}.
\end{proof}

\begin{corollary}
\label{thm:support-monotonicity}
\label{cor:support-supermodular}
Let \(G\) and \(G'\) be directed dependency graphs on the same players.  If
\[
E(G)\subseteq E(G'),
\]
then
\begin{equation}
\label{eq:support-monotonicity}
\mathcal D(G,\mathbf k)\le \mathcal D(G',\mathbf k).
\end{equation}
The inequality is strict if and only if \(G'\) admits a balanced block assignment that is not admissible for \(G\).
Moreover, the degree is supermodular on the Boolean lattice of supports:
\begin{equation}
\label{eq:supermodular}
\mathcal D_{\mathbf k}(E_1)+\mathcal D_{\mathbf k}(E_2)
\le
\mathcal D_{\mathbf k}(E_1\cup E_2)
+\mathcal D_{\mathbf k}(E_1\cap E_2).
\end{equation}
Equivalently, if $E\subseteq E'$ and $e\notin E'$, then
\begin{equation}
\label{eq:increasing-marginal}
\mathcal D_{\mathbf k}(E\cup\{e\})-\mathcal D_{\mathbf k}(E)
\le
\mathcal D_{\mathbf k}(E'\cup\{e\})-\mathcal D_{\mathbf k}(E').
\end{equation}
\end{corollary}

\begin{proof}
The difference in \eqref{eq:support-monotonicity} is the sum of $\mu_{\mathbf k}(F)$ over $F\subseteq E(G')$ with $F\nsubseteq E(G)$, proving monotonicity and its strictness criterion.  The two-point case of \eqref{eq:higher-difference} is equivalent to \eqref{eq:supermodular} and \eqref{eq:increasing-marginal}.
\end{proof}

The effect of a single dependency edge can be resolved more precisely.  Let \(e_j\in\mathbb Z^N\) denote the \(j\)-th standard basis vector, and interpret a coefficient with a negative exponent as zero.

\begin{theorem}
\label{thm:edge-marginal}
Suppose \(j\to i\notin E(G)\), with \(j\ne i\), and let
\[
G^+=G\cup\{j\to i\}.
\]
Then the increase in generic complex algebraic degree is
\begin{equation}
\label{eq:edge-marginal-compact}
\begin{aligned}
\mathcal D(G^+,\mathbf k)-\mathcal D(G,\mathbf k)
=[z^{\mathbf k}]&\left[\bigl(L_i^G(z)+z_j\bigr)^{k_i}
-\bigl(L_i^G(z)\bigr)^{k_i}\right]\\
&\times\prod_{\ell\ne i}\bigl(L_\ell^G(z)\bigr)^{k_\ell}.
\end{aligned}
\end{equation}
Equivalently,
\begin{equation}
\label{eq:edge-marginal-expanded}
\mathcal D(G^+,\mathbf k)-\mathcal D(G,\mathbf k)
=\sum_{r=1}^{k_i}\binom{k_i}{r}
[z^{\mathbf k-r e_j}]
\bigl(L_i^G(z)\bigr)^{k_i-r}
\prod_{\ell\ne i}\bigl(L_\ell^G(z)\bigr)^{k_\ell}.
\end{equation}
In particular, the new edge increases the degree strictly if and only if at least one coefficient on the right-hand side of \eqref{eq:edge-marginal-expanded} is positive.
\end{theorem}

\begin{proof}
Only the support polynomial for player \(i\) changes, from \(L_i^G\) to \(L_i^G+z_j\).  Subtracting the two instances of \eqref{eq:main-permanent} gives \eqref{eq:edge-marginal-compact}.  Applying the binomial theorem and extracting the factor \(z_j^r\) gives \eqref{eq:edge-marginal-expanded}.  Every coefficient involved is nonnegative, which proves the final assertion.
\end{proof}

For a fixed \(r\), the factor \(\binom{k_i}{r}\) chooses the \(r\) equation rows of player \(i\) that use the new block \(j\).  The remaining coefficient counts the balanced completions using old edges.  Thus \eqref{eq:edge-marginal-expanded} partitions the new balanced assignments according to how many times the added dependency edge is used at the player-block level.

\begin{corollary}
\label{cor:edge-hall}
Under the assumptions of \cref{thm:edge-marginal}, the inequality
\[
\mathcal D(G^+,\mathbf k)>\mathcal D(G,\mathbf k)
\]
holds if and only if
\begin{equation}
\label{eq:residual-hall}
\sum_{\ell\in I}\bigl(k_\ell-\mathbf 1_{\{\ell=i\}}\bigr)
\le
\sum_{q\in N_{G^+}^-(I)}
\bigl(k_q-\mathbf 1_{\{q=j\}}\bigr)
\qquad\text{for every }I\subseteq[N].
\end{equation}
\end{corollary}

\begin{proof}
Strict growth is equivalent to the existence of a balanced $\mathbf k$-factor of $G^+$ that uses $j\to i$.  Remove one copy of that arc.  The remaining bipartite matching problem has left capacities $k_\ell-\mathbf 1_{\{\ell=i\}}$ and right capacities $k_q-\mathbf 1_{\{q=j\}}$, with adjacencies still given by $G^+$.  Conversely, any perfect matching for this residual problem becomes a balanced $\mathbf k$-factor using the new arc after the removed copy is restored.  Applying the capacitated Hall argument from \cref{thm:hall-nonvanishing} to these residual capacities gives \eqref{eq:residual-hall}.
\end{proof}

\section{Exact-Support Fibers and Strategy Scaling}
\label{sec:scaling}

The coefficient $\mu_{\mathbf k}(F)$ records the total weight of the balanced factors whose support is exactly $F$.  This section studies the internal geometry of that fiber and then lets all margins grow proportionally.  Alternating-cycle moves are the standard Markov moves for a two-way transportation fiber \cite{DiaconisSturmfels1998}; the point here is that the weights supplied by the player-level expansion have a compatible strict concavity, and that their total has a network-dependent capacity asymptotic.

\subsection{Exact support and alternating cycles}

Write
\[
\mathfrak E_{\mathbf k}^{\circ}(F)
=\{a\in\mathfrak E_{\mathbf k}(F):a_e>0\text{ for every }e\in F\}.
\]
Thus $\mu_{\mathbf k}(F)>0$ exactly when $\mathfrak E_{\mathbf k}^{\circ}(F)$ is nonempty.  Let
\[
d_i^L(F)=|\{j:j\to i\in F\}|,
\qquad
d_j^R(F)=|\{i:j\to i\in F\}|,
\]
and define the residual margins
\begin{equation}
\label{eq:exact-residual-margins}
r_i(F)=k_i-d_i^L(F),
\qquad
s_j(F)=k_j-d_j^R(F).
\end{equation}
For $I\subseteq[N]$, put
\[
N_F^-(I)=\{j:j\to i\in F\text{ for some }i\in I\}.
\]

\begin{theorem}
\label{thm:exact-support-hall}
For an arc set $F\subseteq\mathcal E_N$, the following conditions are equivalent:
\begin{enumerate}[label=\textup{(\roman*)}]
\item $\mu_{\mathbf k}(F)>0$;
\item the residual margins in \eqref{eq:exact-residual-margins} are nonnegative and
\begin{equation}
\label{eq:exact-support-hall}
\sum_{i\in I}r_i(F)
\le
\sum_{j\in N_F^-(I)}s_j(F)
\qquad\text{for every }I\subseteq[N].
\end{equation}
\end{enumerate}
\end{theorem}

\begin{proof}
An exact-support factor can be written uniquely as
\[
a=\mathbf 1_F+b,
\qquad b\in\mathbb Z_{\ge0}^{F}.
\]
The row and column margins of $b$ are respectively $r(F)$ and $s(F)$.  Their total sums agree because both equal $\sum_i k_i-|F|$.  The capacitated Hall theorem applied to the bipartite support $\mathcal B(F)$ gives \eqref{eq:exact-support-hall}.  The corresponding constraint matrix is totally unimodular, so real feasibility and nonnegative integral feasibility coincide.  Adding $\mathbf 1_F$ recovers an element of $\mathfrak E_{\mathbf k}^{\circ}(F)$.
\end{proof}

Let $C$ be a simple cycle of the bipartite graph $\mathcal B(F)$.  Its edges split into two alternating classes $C^+$ and $C^-$.  Define the signed cycle vector
\[
(\gamma_C)_e=
\begin{cases}
1,&e\in C^+,\\
-1,&e\in C^-,\\
0,&e\notin C.
\end{cases}
\]
Adding an integral multiple of $\gamma_C$ preserves every row and column margin.

\begin{proposition}
\label{prop:cycle-fiber-concavity}
Suppose $\mu_{\mathbf k}(F)>0$.
\begin{enumerate}[label=\textup{(\roman*)}]
\item Any two factors in $\mathfrak E_{\mathbf k}^{\circ}(F)$ can be joined by a sequence of simple alternating-cycle moves such that every intermediate factor remains in $\mathfrak E_{\mathbf k}^{\circ}(F)$.
\item Fix $a\in\mathfrak E_{\mathbf k}^{\circ}(F)$ and a simple cycle $C$.  On every integer interval for which $a+t\gamma_C$ remains positive, the sequence
\[
t\longmapsto w_{\mathbf k}(a+t\gamma_C)
\]
is strictly log-concave.  More precisely, whenever the three factors involved are positive,
\begin{equation}
\label{eq:cycle-log-concavity}
\frac{
w_{\mathbf k}(a+(t+1)\gamma_C)
w_{\mathbf k}(a+(t-1)\gamma_C)}
{w_{\mathbf k}(a+t\gamma_C)^2}
=
\prod_{e\in C}
\frac{a_e+t(\gamma_C)_e}
{a_e+t(\gamma_C)_e+1}
<1.
\end{equation}
\end{enumerate}
\end{proposition}

\begin{proof}
Subtract $\mathbf 1_F$ from two exact-support factors.  The difference of the resulting nonnegative tables is an integral circulation on $\mathcal B(F)$.  Direct each edge according to the sign of this difference.  Conservation at every bipartite vertex decomposes the directed difference conformally into alternating simple cycles.  Applying those cycle vectors one at a time never moves a coordinate beyond its value in the target table, so all intermediate residual tables remain nonnegative.  Adding $\mathbf 1_F$ proves part~\textup{(i)}.

For part~\textup{(ii)}, put $b=a+t\gamma_C$.  The numerator $\prod_i k_i!$ in $w_{\mathbf k}$ is constant on the fiber.  For an edge on which $\gamma_C=1$, the two neighboring factorial ratios contribute $b_e/(b_e+1)$; the same expression results on an edge on which $\gamma_C=-1$.  Multiplication over the cycle gives \eqref{eq:cycle-log-concavity}.
\end{proof}

The proposition does not assert global log-concavity of the network degree polynomial.  It identifies strict log-concavity along the cycle directions that generate the toric fiber.

\subsection{Dilations of exact-support fibers}

For $m\ge1$, define the unweighted exact-support count
\[
\nu_{F,\mathbf k}(m)
=|\mathfrak E_{m\mathbf k}^{\circ}(F)|.
\]
The distinction between $\nu_{F,\mathbf k}(m)$ and $\mu_{m\mathbf k}(F)$ is important: the former counts exponent vectors, whereas the latter sums their factorial weights.

\begin{theorem}
\label{thm:exact-support-ehrhart}
Suppose $\mu_{\mathbf k}(F)>0$, let $c(F)$ be the number of connected components of $\mathcal B(F)$, and put
\begin{equation}
\label{eq:cycle-rank}
q(F)=|F|-2N+c(F).
\end{equation}
Then $\nu_{F,\mathbf k}(m)$ is the interior Ehrhart polynomial of $\mathcal T_F(\mathbf k)$.  In particular, it is a polynomial in $m$ of degree $q(F)$ and
\begin{equation}
\label{eq:exact-support-reciprocity}
\nu_{F,\mathbf k}(m)
=(-1)^{q(F)}L_{F,\mathbf k}(-m),
\end{equation}
where
\[
L_{F,\mathbf k}(m)
=|m\mathcal T_F(\mathbf k)\cap\mathbb Z^F|.
\]
Its leading coefficient is the relative lattice volume of $\mathcal T_F(\mathbf k)$.
\end{theorem}

\begin{proof}
The margin equations give
\[
\mathcal T_F(m\mathbf k)=m\mathcal T_F(\mathbf k).
\]
The transportation constraint matrix is totally unimodular, so $\mathcal T_F(\mathbf k)$ is a lattice polytope.  Because it contains a point positive on every edge of $F$, its relative interior consists exactly of the points whose coordinates indexed by $F$ are all positive.  Hence $\nu_{F,\mathbf k}(m)$ counts the relative-interior lattice points of the $m$th dilation.  Ehrhart--Macdonald reciprocity gives \eqref{eq:exact-support-reciprocity} \cite{BeckRobins2015}.  The dimension is $q(F)$ by \eqref{eq:transportation-dimension}, which proves the degree and leading-coefficient statements.
\end{proof}

Consequently, the minimal positive supports are precisely the exact supports that remain rigid under proportional dilation: if $\mathcal B(F)$ is a forest, then $\nu_{F,\mathbf k}(m)=1$ for every $m\ge1$; if it contains a cycle $C$, then $2a-\gamma_C$, $2a$, and $2a+\gamma_C$ give three distinct exact-support factors at $m=2$.

\subsection{Network capacity and maximum-entropy flow}

For an arc set $F$, set
\[
L_i^F(z)=\sum_{j:j\to i\in F}z_j,
\qquad
P_{F,\mathbf k}(z)=\prod_{i=1}^N L_i^F(z)^{k_i},
\]
and define
\begin{equation}
\label{eq:network-capacity}
\operatorname{cap}_{\mathbf k}(F)
=\inf_{z\in\mathbb R_{>0}^N}
\frac{P_{F,\mathbf k}(z)}{z_1^{k_1}\cdots z_N^{k_N}}.
\end{equation}
For $p\in\mathcal T_F(\mathbf k)$, let
\[
q_{ji}(p)=\frac{p_{ji}}{k_i}
\qquad(j\to i\in F)
\]
and write
\[
H(q_i(p))=-\sum_{j:j\to i\in F}q_{ji}(p)\log q_{ji}(p),
\]
with $0\log0=0$.

Suppose $\mathcal D_{\mathbf k}(E)>0$.  Define the \emph{essential support} of $E$ at margin $\mathbf k$ by
\begin{equation}
\label{eq:essential-support}
E_{\mathbf k}^*(E)
=\{e\in E:p_e>0\text{ for some }p\in\mathcal T_E(\mathbf k)\}.
\end{equation}

\begin{proposition}
\label{prop:essential-support}
Let $F^*=E_{\mathbf k}^*(E)$.  Restriction to $F^*$ identifies $\mathcal T_E(\mathbf k)$ with $\mathcal T_{F^*}(\mathbf k)$, and the latter polytope contains a point positive on every edge of $F^*$.  Moreover, for every integer $m\ge1$,
\begin{equation}
\label{eq:essential-degree-reduction}
\mathcal D_{m\mathbf k}(E)=\mathcal D_{m\mathbf k}(F^*).
\end{equation}
\end{proposition}

\begin{proof}
Every coordinate outside $F^*$ vanishes at every point of $\mathcal T_E(\mathbf k)$, so restriction gives the first identification.  For each $e\in F^*$, choose $p^{(e)}\in\mathcal T_E(\mathbf k)$ with $p^{(e)}_e>0$.  The average of these finitely many flows restricts to a point of $\mathcal T_{F^*}(\mathbf k)$ positive on every edge.

For the last assertion, $\mathcal T_E(m\mathbf k)=m\mathcal T_E(\mathbf k)$.  Hence every flow with margins $m\mathbf k$ also vanishes outside $F^*$, while every flow on $F^*$ extends by zero to $E$.  Restriction therefore bijects the integral balanced factors at margins $m\mathbf k$ and preserves their factorial weights, proving \eqref{eq:essential-degree-reduction}.
\end{proof}

Assume $\mathcal T_F(\mathbf k)\cap\mathbb R_{>0}^F\ne\varnothing$.  Let $F_1,\ldots,F_{c(F)}$ be the connected components of $\mathcal B(F)$ and choose one right vertex $j_h^R$ from each component.  Put
\[
R=[N]\setminus\{j_1,\ldots,j_{c(F)}\}.
\]
For a probability vector $q_i$ on the right neighbors of $i_L$, extend $q_i$ by zero to $[N]$.  For a positive flow $p\in\mathcal T_F(\mathbf k)$, set
\begin{equation}
\label{eq:full-covariance}
\Sigma(p)
=\sum_{i=1}^Nk_i
\left(\operatorname{Diag}(q_i(p))-q_i(p)q_i(p)^{\mathsf T}\right),
\end{equation}
and let $\widehat\Sigma(p)$ be its principal submatrix on $R$.  Once the maximizing flow $p^*$ below has been defined, write
\begin{equation}
\label{eq:reduced-covariance}
\widehat\Sigma=\widehat\Sigma(p^*).
\end{equation}

\begin{lemma}
\label{lem:lattice-normalization}
Let $p\in\mathcal T_F(\mathbf k)\cap\mathbb R_{>0}^F$.  For a connected component $F_h$, denote its right vertex set by $J_h$.  Then
\begin{enumerate}[label=\textup{(\roman*)}]
\item
\[
\ker\Sigma(p)
=\operatorname{span}_{\mathbb R}\{\mathbf 1_{J_h}:1\le h\le c(F)\}.
\]
\item If $X(p)$ is the right-count vector obtained from one block of draws, consisting of $k_i$ independent draws from the right neighbors of each $i_L$ with probabilities $q_i(p)$, then its support-difference lattice is
\[
\Lambda_F
=\operatorname{span}_{\mathbb Z}\{x-y:x,y\in\operatorname{supp}X(p)\}
=\left\{v\in\mathbb Z^N:\sum_{j\in J_h}v_j=0\text{ for every }h\right\}.
\]
Deleting the coordinates $j_1,\ldots,j_{c(F)}$ identifies $\Lambda_F$ unimodularly with $\mathbb Z^{N-c(F)}$.
\item The matrix $\widehat\Sigma(p)$ is positive definite.  Its determinant is independent of the chosen right vertex deleted from each connected component.
\end{enumerate}
\end{lemma}

\begin{proof}
For $v\in\mathbb R^N$, the quadratic form of \eqref{eq:full-covariance} is
\[
v^{\mathsf T}\Sigma(p)v
=\sum_{i=1}^Nk_i
\left(
\sum_{j:j\to i\in F}q_{ji}(p)v_j^2
-\left(\sum_{j:j\to i\in F}q_{ji}(p)v_j\right)^2
\right).
\]
Because $p$ is positive on $F$, this expression vanishes exactly when $v$ is constant on the right neighbor set of every left vertex.  In each connected component of $\mathcal B(F)$, any two right vertices are joined by a sequence of right vertices in which consecutive vertices share a left neighbor.  Hence $v$ is constant on each $J_h$, proving part~\textup{(i)}.

Every realization of $X(p)$ has the same coordinate sum on $J_h$, so its support differences lie in the lattice displayed in part~\textup{(ii)}.  Conversely, whenever $j$ and $\ell$ share a left neighbor $i_L$, two realizations that differ in one of the $k_i$ draws have difference $e_j-e_\ell$.  Such differences along a connected right projection of $F_h$ generate the root lattice
\[
\left\{v\in\mathbb Z^{J_h}:\sum_{j\in J_h}v_j=0\right\}.
\]
Taking the direct sum over the components proves the asserted equality.  The basis $e_j-e_{j_h}$, with $j\in J_h\setminus\{j_h\}$, shows that coordinate deletion is a unimodular lattice isomorphism.

Part~\textup{(i)} and coordinate deletion imply that $\widehat\Sigma(p)$ is positive definite.  The full covariance is block diagonal over the sets $J_h$.  On each block it is symmetric, has rank $|J_h|-1$, and has kernel spanned by the all-ones vector.  Its adjugate is therefore a scalar multiple of $\mathbf 1\mathbf 1^{\mathsf T}$, so all principal cofactors of that block are equal.  The determinant after deleting one coordinate per component is the product of these cofactors and is consequently independent of the deleted vertices.
\end{proof}

\Needspace{12\baselineskip}
\begin{theorem}
\label{thm:capacity-asymptotic}
Suppose $\mathcal T_F(\mathbf k)\cap\mathbb R_{>0}^F\ne\varnothing$.
\begin{enumerate}[label=\textup{(\roman*)}]
\item The capacity has the entropy representation
\begin{equation}
\label{eq:capacity-entropy}
\log\operatorname{cap}_{\mathbf k}(F)
=\max_{p\in\mathcal T_F(\mathbf k)}
\sum_{i=1}^Nk_iH(q_i(p)).
\end{equation}
The maximizing flow $p^*$ is unique and positive on every edge of $F$.  The infimum in \eqref{eq:network-capacity} is attained at a vector $z^*>0$, unique up to multiplication by an independent positive scalar on the right vertex set of each connected component of $\mathcal B(F)$, and
\begin{equation}
\label{eq:maximum-entropy-scaling}
p_{ji}^*
=\frac{k_i z_j^*}{L_i^F(z^*)}
\qquad(j\to i\in F).
\end{equation}
\item The reduced covariance matrix in \eqref{eq:reduced-covariance} is positive definite, its determinant is independent of the deleted right vertices, and
\begin{equation}
\label{eq:capacity-degree-asymptotic}
\mathcal D_{m\mathbf k}(F)
\sim
\frac{\operatorname{cap}_{\mathbf k}(F)^m}
{(2\pi m)^{(N-c(F))/2}\sqrt{\det\widehat\Sigma}}.
\end{equation}
\item Exact-support factors asymptotically carry the entire degree:
\begin{equation}
\label{eq:exact-support-dominance}
\frac{\mu_{m\mathbf k}(F)}{\mathcal D_{m\mathbf k}(F)}
\longrightarrow1.
\end{equation}
\end{enumerate}
\end{theorem}

\begin{proof}
For $p\in\mathcal T_F(\mathbf k)$ and $z>0$, weighted arithmetic--geometric mean gives, for each $i$,
\[
\log L_i^F(z)
\ge
\sum_{j:j\to i\in F}
q_{ji}(p)\log\frac{z_j}{q_{ji}(p)}.
\]
After multiplication by $k_i$ and summation over $i$, the column-margin equations imply
\[
\log\frac{P_{F,\mathbf k}(z)}{z^{\mathbf k}}
\ge
\sum_i k_iH(q_i(p)).
\]
The entropy functional on the right equals
\[
-\sum_{j\to i\in F}p_{ji}\log\frac{p_{ji}}{k_i}.
\]
It is strictly concave on the transportation polytope, hence has a unique maximizer $p^*$.  Since the polytope contains a positive point, mixing that point with a boundary point shows from the infinite one-sided derivative of $-x\log x$ at zero that $p^*$ is positive on every edge.  The first-order condition in the relative interior, equivalently Lagrange multipliers for the affine margin equations, gives positive vectors $u$ and $z^*$ such that $p_{ji}^*=u_i z_j^*$.  The row margins determine $u_i=k_i/L_i^F(z^*)$, proving \eqref{eq:maximum-entropy-scaling}.  Equality holds in the weighted arithmetic--geometric mean inequality at this scaling, which proves \eqref{eq:capacity-entropy} and shows that the capacity infimum is attained.  If $\widetilde z$ is any minimizing vector, equality must hold in every weighted arithmetic--geometric mean inequality above with $p=p^*$; hence \eqref{eq:maximum-entropy-scaling} also holds with $\widetilde z$.  The ratios $p_{ji}^*/p_{\ell i}^*=z_j^*/z_\ell^*$ then determine all ratios among right vertices in each connected component of $\mathcal B(F)$.  Thus any two minimizing vectors differ by an independent positive scalar on each component; conversely, such rescaling leaves the objective unchanged because the two margins have equal total on each component.  This is the usual entropy--matrix-scaling duality \cite{Elfving1980}.

We next convert the coefficient into a lattice probability.  For every $i$, take $mk_i$ independent draws from the right neighbors of $i_L$, with
\[
\mathbb P\{J_{i,r}=j\}=q_{ji}(p^*).
\]
Group the draws into $m$ independent blocks, each containing $k_i$ draws of type $i$.  Let $A^{(m)}=(A_{ji}^{(m)})_{j\to i\in F}$ be the resulting table of edge counts and let $S_j^{(m)}=\sum_iA_{ji}^{(m)}$ be its right-margin vector.  A table $a$ with row margins $m\mathbf k$ has probability
\[
\prod_i\frac{(mk_i)!}{\prod_j a_{ji}!}
\prod_{j\to i\in F}q_{ji}(p^*)^{a_{ji}}.
\]
If its column margins also equal $m\mathbf k$, then \eqref{eq:maximum-entropy-scaling} makes the second product independent of $a$:
\[
\prod_{j\to i\in F}q_{ji}(p^*)^{a_{ji}}
=
\frac{(z^*)^{m\mathbf k}}
{P_{F,\mathbf k}(z^*)^m}
=\operatorname{cap}_{\mathbf k}(F)^{-m}.
\]
Summing over the fiber and using \eqref{eq:eulerian-degree} yields the exact identity
\begin{equation}
\label{eq:degree-lattice-probability}
\mathcal D_{m\mathbf k}(F)
=\operatorname{cap}_{\mathbf k}(F)^m
\mathbb P\{S^{(m)}=m\mathbf k\}.
\end{equation}
The same calculation restricted to tables with support exactly $F$ gives
\begin{equation}
\label{eq:exact-support-lattice-probability}
\mu_{m\mathbf k}(F)
=\operatorname{cap}_{\mathbf k}(F)^m
\mathbb P\{S^{(m)}=m\mathbf k,\ A_e^{(m)}>0\text{ for every }e\in F\}.
\end{equation}

The mean of $S^{(m)}$ is $m\mathbf k$, and the covariance of its coordinates in $R$ is $m\widehat\Sigma$.  More precisely, the restriction of $S^{(m)}$ to $R$ is a sum of $m$ independent copies of the reduced one-block vector in \cref{lem:lattice-normalization}.  These vectors have finite support and positive-definite covariance $\widehat\Sigma$.  By part~\textup{(ii)} of that lemma, their support-difference lattice after coordinate deletion is the full lattice $\mathbb Z^{N-c(F)}$; thus the reduced distribution is aperiodic with maximal span one.  Its integral mean is the restriction of $\mathbf k$ to $R$.  Moreover, the omitted count in each component is determined by the retained counts and the fixed component total, so equality of the retained coordinates is equivalent to $S^{(m)}=m\mathbf k$.  If $N-c(F)>0$, the multivariate lattice local central limit theorem \cite[Chapter~5]{BhattacharyaRao2010} therefore gives
\begin{equation}
\label{eq:lattice-local-limit}
\mathbb P\{S^{(m)}=m\mathbf k\}
\sim
\frac{1}{(2\pi m)^{(N-c(F))/2}\sqrt{\det\widehat\Sigma}}.
\end{equation}
If $N-c(F)=0$, the probability in \eqref{eq:lattice-local-limit} is identically one, and the same formula holds with the determinant of the empty matrix equal to one.
The coordinate deletion in \cref{lem:lattice-normalization} is a lattice isomorphism, so \eqref{eq:lattice-local-limit} is expressed with respect to ordinary counting measure on $\mathbb Z^{N-c(F)}$ and contains no additional covolume factor.
Combining \eqref{eq:degree-lattice-probability} and \eqref{eq:lattice-local-limit} proves part~\textup{(ii)}.

Finally, dividing \eqref{eq:exact-support-lattice-probability} by \eqref{eq:degree-lattice-probability} gives the conditional-probability identity
\begin{equation}
\label{eq:exact-support-conditional-probability}
\frac{\mu_{m\mathbf k}(F)}{\mathcal D_{m\mathbf k}(F)}
=\mathbb P\{A_e^{(m)}>0\text{ for every }e\in F\mid S^{(m)}=m\mathbf k\}.
\end{equation}
An edge $j\to i$ is absent from the unconditioned sampled table with probability
\[
(1-q_{ji}(p^*))^{mk_i}.
\]
If $q_{ji}(p^*)=1$, this probability is zero; otherwise it decays exponentially because $q_{ji}(p^*)>0$.  Since $F$ is finite, a union bound gives
\[
\mathbb P\{A_e^{(m)}=0\text{ for some }e\in F\}=O(e^{-\delta m})
\]
for some $\delta>0$.  In contrast, \eqref{eq:lattice-local-limit} is of order $m^{-(N-c(F))/2}$.  Conditioning in \eqref{eq:exact-support-conditional-probability} therefore shows that the probability of omitting an edge is $O(m^{(N-c(F))/2}e^{-\delta m})$, which proves \eqref{eq:exact-support-dominance}.
\end{proof}

\begin{corollary}
\label{cor:essential-support-asymptotic}
Suppose $\mathcal D_{\mathbf k}(E)>0$, let $F^*=E_{\mathbf k}^*(E)$ be defined by \eqref{eq:essential-support}, and form $p^*$ and $\widehat\Sigma^*$ from $F^*$ as in \cref{thm:capacity-asymptotic}.  Then
\begin{equation}
\label{eq:essential-support-asymptotic}
\mathcal D_{m\mathbf k}(E)
\sim
\frac{\operatorname{cap}_{\mathbf k}(F^*)^m}
{(2\pi m)^{(N-c(F^*))/2}\sqrt{\det\widehat\Sigma^*}},
\end{equation}
and
\begin{equation}
\label{eq:essential-exact-dominance}
\frac{\mu_{m\mathbf k}(F^*)}{\mathcal D_{m\mathbf k}(E)}
\longrightarrow1.
\end{equation}
\end{corollary}

\begin{proof}
By \cref{prop:essential-support}, the polytope $\mathcal T_{F^*}(\mathbf k)$ contains a point positive on all of $F^*$ and its degree agrees with the degree on $E$ at every dilation.  Apply \cref{thm:capacity-asymptotic} to $F^*$.
\end{proof}

The essential-support reduction is needed even for complete dependency graphs.  For the complete loopless graph on three players and $\mathbf k=(1,2,3)$, every feasible flow is supported on
\[
F^*=\{3\to1,\ 3\to2,\ 1\to3,\ 2\to3\}.
\]
Here $c(F^*)=2$, $\operatorname{cap}_{\mathbf k}(F^*)=27/4$, and direct coefficient extraction gives
\[
\mathcal D_{m\mathbf k}(\mathcal E_3)
=\mathcal D_{m\mathbf k}(F^*)
=\binom{3m}{m}
\sim
\frac{\sqrt3}{2\sqrt{\pi m}}
\left(\frac{27}{4}\right)^m.
\]
Thus the power of $m$ is determined by $c(F^*)$, not by the connected complete support.

In particular, for every positive-degree support $E$,
\begin{equation}
\label{eq:capacity-growth-rate}
\lim_{m\to\infty}\mathcal D_{m\mathbf k}(E)^{1/m}
=\operatorname{cap}_{\mathbf k}(E_{\mathbf k}^*(E)),
\end{equation}
and the same limit holds with $\mu_{m\mathbf k}(E_{\mathbf k}^*(E))$ in place of the degree.  Formula \eqref{eq:maximum-entropy-scaling} also makes the exponential rate accessible through a standard matrix-scaling problem rather than enumeration of the expanded polynomial graph.

\subsection{Comparison with complete-support asymptotics}
\label{sec:known-asymptotics}

Take $F=\mathcal E_N$, $N\ge3$, and $\mathbf k=\mathbf1$.  Symmetry gives $z^*=\mathbf1$, $p^*_{ji}=1/(N-1)$ for $j\ne i$, and
\[
\operatorname{cap}_{\mathbf1}(\mathcal E_N)=(N-1)^N,
\qquad
\Sigma=\frac{N(N-2)}{(N-1)^2}
\left(I-\frac1N\mathbf1\mathbf1^{\mathsf T}\right).
\]
Thus $c(F)=1$ and
\[
\det\widehat\Sigma
=\frac1N\left(\frac{N(N-2)}{(N-1)^2}\right)^{N-1}.
\]
Substitution into \eqref{eq:capacity-degree-asymptotic} yields
\begin{equation}
\label{eq:complete-diagonal-asymptotic}
\mathcal D(\mathcal E_N,m\mathbf1)
\sim
\frac{\sqrt N\,(N-1)^{Nm+N-1}}
{\bigl(2\pi mN(N-2)\bigr)^{(N-1)/2}}.
\end{equation}
This is precisely the block-derangement asymptotic of \cite[Theorem~6.1]{Vidunas2017}: its block size is $m$, whereas each player here has $m+1$ pure strategies.  Keeping this shift explicit also reconciles \eqref{eq:complete-diagonal-asymptotic} with the strategy-count notation of \cite[Remark~2.12]{AboPortakalSodomaco2026}.

Unequal proportional growth on complete supports is also treated in \cite[Section~6]{Vidunas2017}.  The distinction in \cref{cor:essential-support-asymptotic} is the treatment of arbitrary structural zeros through $F^*$, which determines both the covariance rank and the correct lattice normalization.  The example $\mathbf k=(1,2,3)$ above shows why the support reduction changes even the power of $m$.  The exact-support concentration in \eqref{eq:essential-exact-dominance} further identifies which arc sets account for the leading weighted count.

\subsection{A three-player fiber}

Let $F=\mathcal E_3$ be the complete loopless directed graph on three players and take $\mathbf k=(2,2,2)$.  Every balanced factor of margins $m\mathbf k$ is determined by an integer $0\le t\le2m$: one directed $3$-cycle has multiplicity $t$ and the reverse cycle has multiplicity $2m-t$.  Hence
\begin{equation}
\label{eq:franel-degree}
\mathcal D_{m\mathbf k}(F)
=\sum_{t=0}^{2m}\binom{2m}{t}^3,
\qquad
\nu_{F,\mathbf k}(m)=2m-1.
\end{equation}
The bipartite support is a $6$-cycle, so $q(F)=1$ and $c(F)=1$.  The maximum-entropy flow assigns value $1$ to every arc, giving $\operatorname{cap}_{\mathbf k}(F)=64$.  Deleting one right coordinate gives $\det\widehat\Sigma=3/4$, and \cref{thm:capacity-asymptotic} yields
\begin{equation}
\label{eq:franel-asymptotic}
\mathcal D_{m\mathbf k}(F)
\sim
\frac{64^m}{\pi\sqrt3\,m}.
\end{equation}
The summands in \eqref{eq:franel-degree} form the one-dimensional instance of \eqref{eq:cycle-log-concavity}; the two boundary terms have weight one and are asymptotically negligible.

\subsection{A sparse asymmetric four-player network}
\label{sec:sparse-example}

Let $\mathbf k=(3,6,3,6)$ and specify the dependency support $F$ by the receiver--row matrix
\[
A_F=\begin{pmatrix}
0&1&1&0\\0&0&1&1\\1&0&0&1\\1&1&0&0
\end{pmatrix}.
\]
This network has eight of the twelve possible arcs and is not invariant under reversal.  Its bipartite support is a connected $8$-cycle, so $c(F)=1$ and $q(F)=1$.
\begin{figure}[htbp]
\centering
\begin{tikzpicture}[>=Stealth, every node/.style={circle,draw,minimum size=7mm},scale=1.05]
\node (1) at (0,2) {$1$}; \node (2) at (3,2) {$2$};
\node (3) at (3,0) {$3$}; \node (4) at (0,0) {$4$};
\draw[->] (2)--(1); \draw[->] (3)--(2); \draw[->] (4)--(3); \draw[->] (1)--(4);
\draw[<->] (1)--(3); \draw[<->] (2)--(4);
\end{tikzpicture}
\caption{The sparse support $F$ for $\mathbf k=(3,6,3,6)$.  Arrows indicate payoff dependence $j\to i$; the crossing of the diagonals is not a vertex.  Every displayed arc is essential.}
\label{fig:sparse-four-player}
\end{figure}

In receiver--row notation, all balanced factors at margins $m\mathbf k$ are
\begin{equation}
\label{eq:sparse-fiber}
T_m(t)=\begin{pmatrix}
0&t&3m-t&0\\
0&0&t&6m-t\\
3m-t&0&0&t\\
t&6m-t&0&0
\end{pmatrix},\qquad 0\le t\le3m.
\end{equation}
Indeed, the row and column equations determine every entry from the $(1,2)$ entry $t$.  The point $T_1(2)$ is positive on all eight arcs, proving $F^*=F$.  The two endpoint supports are transportation forests, hence inclusion-minimal positive supports; the interior integer values give $\nu_{F,\mathbf k}(m)=3m-1$.  The weighted formula is
\begin{equation}
\label{eq:sparse-exact-degree}
\mathcal D(F,m\mathbf k)=\sum_{t=0}^{3m}\binom{3m}{t}^{\!2}\binom{6m}{t}^{\!2},
\qquad
\mu_{m\mathbf k}(F)=\mathcal D(F,m\mathbf k)-1-\binom{6m}{3m}^{\!2}.
\end{equation}

The scaling vector $z^*=(1,2,1,2)$ has $L_i^F(z^*)=3$ for every $i$, and \eqref{eq:maximum-entropy-scaling} gives $p^*=T_1(2)$.  Consequently,
\[
\operatorname{cap}_{\mathbf k}(F)=\frac{3^{18}}{2^{12}},\qquad
\widehat\Sigma=
\begin{pmatrix}
2&-4/3&0\\
-4/3&2&-2/3\\
0&-2/3&2
\end{pmatrix},\qquad
\det\widehat\Sigma=\frac{32}{9},
\]
where the fourth right coordinate is deleted.  The first-order prediction is therefore
\begin{equation}
\label{eq:sparse-asymptotic}
\mathcal D(F,m\mathbf k)\sim A_m,
\qquad A_m=\frac{3}{16(\pi m)^{3/2}}
\left(\frac{3^{18}}{2^{12}}\right)^m.
\end{equation}
Table~\ref{tab:sparse-asymptotic} compares this prediction with the finite sum in \eqref{eq:sparse-exact-degree}.  The first three dilations were also checked by the independent truncated coefficient algorithm of \cref{prop:player-dp}.  The example displays a one-dimensional transportation fiber but a three-dimensional count fluctuation: $q(F)=1$ controls the unweighted Ehrhart count, while $N-c(F)=3$ controls the power in the weighted degree asymptotic.

\begin{table}[htbp]
\centering
\caption{Exact finite-sum counts and their asymptotic approximation for the sparse four-player network.  Counts for $m\ge5$ are rounded to eight significant digits; the ratios are computed from unrounded integers.  The last column is the proportion contributed by the two proper endpoint supports.}
\label{tab:sparse-asymptotic}
\begin{tabular}{r r r r}
\hline
$m$ & $\mathcal D(F,m\mathbf k)$ & $\mathcal D(F,m\mathbf k)/A_m$ & $1-\mu_{m\mathbf k}(F)/\mathcal D(F,m\mathbf k)$\\
\hline
1 & $2750$ & 0.863443 & $1.458\times10^{-1}$ \\
2 & $98911190$ & 0.928686 & $8.632\times10^{-3}$ \\
5 & $2.2139173\times10^{22}$ & 0.971017 & $1.087\times10^{-6}$ \\
10 & $6.0135390\times10^{46}$ & 0.985433 & $2.326\times10^{-13}$ \\
20 & $1.2274431\times10^{96}$ & 0.992698 & $7.605\times10^{-27}$ \\
\hline
\end{tabular}

\end{table}

\section{Support Decomposition}
\label{sec:basic_structures}

The directed cycle already illustrates the balance constraint.  Its degree is zero unless all $k_i$ are equal.  If $k_i=k$ for every $i$, there is one balanced factor, with multiplicity $k$ on each cycle arc, and its weight is one; hence $\mathcal D=1$.

\subsection{Strongly connected decomposition}

For general graphs, the strongly connected components expose a block structure in the polynomial matrix.

\begin{theorem}
\label{thm:SCC}
Let $G$ be the player interaction graph, and let $C_1,\dots,C_r$ be its strongly connected components.  For each component, let \(\mathcal D(C_\ell)\) denote the degree determined by the diagonal equation--variable block supported on players in \(C_\ell\).  Then
\[
\mathcal{D}(G)=\prod_{\ell=1}^r \mathcal{D}(C_\ell).
\]
\end{theorem}

\begin{proof}
Topologically order the acyclic condensation of $G$.  The corresponding simultaneous ordering of equation and variable groups makes the structure matrix $M$ block upper triangular, with diagonal blocks $M_{\ell\ell}$ belonging to the induced component systems.  Every nonzero term in the permanent preserves these blocks successively, and therefore
\[
\perm(M)=\prod_{\ell=1}^r\perm(M_{\ell\ell}).
\]
Apply \cref{thm:Datta} to $M$ and to each diagonal block.
\end{proof}

\paragraph{Complex and feasible counts.}
All formulas above count isolated complex-torus roots for coefficients in a Zariski-open dense set, with algebraic multiplicity.  For a fixed payoff realization, the number of real roots can be smaller, and a real root is a feasible totally mixed equilibrium only when every independent probability and every eliminated probability $1-\sum_{s=1}^{k_i}x_{i,s}$ is strictly positive.  Reality, feasibility, dynamical stability, and sensitivity are payoff- or model-dependent and do not follow from the generic complex degree.

\section{Conclusion}
\label{sec:conclusion}

The edge-marked polynomial $\Phi_{\mathbf k}$ yields two coherent regimes.  At fixed margins, balanced factors give the complexity boundary, capacitated Hall criterion, sparse positive cores, and support comparative statics; under proportional dilation, the same fibers give the capacity and maximum-entropy asymptotic on the essential support.  Strongly connected components factorize the degree.  These results concern generic isolated complex roots; real feasibility and dynamics require additional payoff-specific analysis.
\section*{Declaration on the Use of AI Tools}

During the preparation of this manuscript, the authors used OpenAI Codex to assist with literature-search support, manuscript and proof review, language and \LaTeX{} editing, and computational consistency checks. All AI-assisted outputs, mathematical arguments, citations, and computations were independently reviewed and verified by the authors. The authors assume responsibility for all content.
\enlargethispage{\baselineskip}
\bibliographystyle{siamplain}
\bibliography{ref}

\end{document}